\documentclass[11pt]{amsart}
\usepackage{amsmath}
\usepackage{amssymb,amscd}
\usepackage{graphicx}

\numberwithin{equation}{section}

\newtheorem{defn}{Definition}[section]
\newtheorem{theorem}{Theorem}[section]

\newtheorem{corollary}[theorem]{Corollary}

\theoremstyle{definition}

\newtheorem{lemma}[theorem]{Lemma}

\newtheorem{prop}[theorem]{Proposition}
\newtheorem{remark}[theorem]{REMARK}

\def \begineq{\begin{equation}}
\def \endeq{\end{equation}}

\def \bb{\mathbb}

\def \CC{{\bb{C}}}

\def \RR{{\bb{R}}}

\def \ZZ{{\bb{Z}}}

\def \({\left(}
\def \){\right)}
\def \<{\langle}
\def \>{\rangle}

\def \ds{\displaystyle}

\begin{document}

\title{Equivariant principal bundles on nonsingular toric varieties}
 %MSC classification
 %53C15; 53D20

\author[A. Dey]{Arijit Dey}

\address{Department of Mathematics, Indian Institute of Technology-Madras, Chennai, India }

\email{arijitdey@gmail.com}

\author[M. Poddar]{Mainak Poddar}

\address{Departamento de
Matem\'aticas, Universidad de los Andes, Bogot\'a, Colombia}

\email{mainakp@gmail.com}

\subjclass[2010]{32L05,14M25}

\keywords{equivariant bundles, principal bundles, toric varieties}

%\thanks{Second author was supported by the Programa de Apoyo a Profesores Asistentes  from
%Universidad de los Andes.}

\abstract We give a classification of the equivariant principal
$G$-bundles on a nonsingular toric variety when $G$ is a closed
Abelian subgroup of $GL_k(\CC)$. We prove that any such bundle splits, 
that is, admits a reduction of structure group to the intersection of $G$
with a torus. We give an explicit parametrization of the isomorphism classes of such bundles 
for a large family of $G$ when $X$ is complete.
\endabstract

\maketitle

\section{Introduction}
 Denote the algebraic torus $(\CC^*)^n$ by $T$. A $T$-equivariant principal
  $G$-bundle on a complex manifold $X$
is a locally trivial, principal $G$-bundle $\pi: \mathcal{E} \to
X$ such that $\mathcal{E}$ and $X$ are left $T$-spaces, the map
$\pi$ is $T$-equivariant and the actions of $T$ and $G$ commute:
$$t(e\cdot g)= (te)\cdot g  \; {\rm for\; all\;} t\in T, \; g \in G
 \; {\rm and}\; e\in
\mathcal{E}.$$
 If, in addition, the bundle $\pi: \mathcal{E} \to X
$ and the actions of $T$ and $G$ are holomorphic we say that
$\mathcal{E}$ is a holomorphic $T$-equivariant principal
$G$-bundle.

Let $X_{\Xi}$ be a complete nonsingular toric variety of dimension
$n$ corresponding to a fan $\Xi$. Denote the set of
$d$-dimensional cones in $\Xi$ by $\Xi(d)$. For a cone $\sigma$ in
$\Xi$, denote the corresponding affine variety by $X_{\sigma}$ and
the corresponding  $T$-orbit by $O_{\sigma}$. Note that each orbit
$O_{\sigma}$ has a natural group structure and the principal orbit
$O$ is identified with $T$ (see \cite{Oda}, Proposition 1.6).
 Let $T_{\sigma}$
denote the stabilizer of any point in $O_{\sigma}$. Then $T$
admits a decomposition, $T \cong T_{\sigma} \times O_\sigma$. Let
$\pi_{\sigma}: T \to T_{\sigma} $ be the associated projection.

Let $G$ be an Abelian subgroup of $GL(k,\CC)$ for some positive
integer $k$. Note that a classification of such groups is given in
\cite{AM}, Proposition 2.3. Assume further that $G$ is closed in
$GL_k(\CC)$. Then our main theorem is the following (same as
Theorem \ref{thm:clas}).

\begin{theorem}\label{thm:main}  The isomorphism classes of holomorphic
 $T$-equivariant principal $G$-bundles on $X_{\Xi}$ are in one-to-one
correspondence with collections of holomorphic group homomorphisms
$\{\rho_{\sigma}: T_{\sigma} \to G \,|\, \sigma \; {\rm is \; a\;
maximal \; cone\; in\;} \Xi \}$ which satisfy the extension
condition: Each $ (\rho_{\tau} \circ \pi_{\tau})  (\rho_{\sigma}
\circ \pi_{\sigma} )^{-1}$ extends to a $G$-valued holomorphic
function over $X_{\sigma} \cap X_{\tau}$.
\end{theorem}

A similar classification for algebraic $T$-equivariant bundles
over $X_{\Xi}$ is given in Theorem \ref{thm:clasA}.
 These theorems provide a partial analogue to Klyachko's
classification of vector bundles on toric varieties in \cite{Kly}.

In particular, we show that any $T$-equivariant principal $G$-bundle on a nonsingular
affine toric variety is trivial if $G$ is a closed linear Abelian group (see Lemma \ref{lem:cone1}).    
The main tool used by Klyachko for the local  classification
(i.e. classification on affine toric variety) is representation theory.
In analogous situation, our main tool is complex analysis. However, our 
approach does not work when $G$ is not Abelian since the proof of Theorem 
\ref{thm:homo} relies heavily on this assumption. In this article, we are restricted to the case of 
nonsingular toric variety because we use Oka-Grauert theory.

As a corollary of the main theorem, we prove that any such $T$-equivariant principal
$G$-bundle splits, 
that is, admits a reduction of structure group to the intersection of $G$
with a torus, see Theorems \ref{thm:last1} and \ref{thm:last2}.
 In Corollary \ref{cor:last}, we give an explicit parametrization of the isomorphism classes of
$T$-equivariant principal $G$-bundles,
for a large family of $G$, when $\Xi$ is complete.
This generalizes the classical description of isomorphism classes of $T$-equivariant
line bundles.

 We also prove that if $G$ is any discrete group, then any
holomorphic $T$-equivariant principal $G$-bundle on $X_{\Xi}$ is
trivial with trivial $T$-action (Theorem \ref{thm:discrete}).
Our
method may be used to prove a similar result for
$\Gamma^n$-equivariant principal $G$-bundles over a topological
toric manifold \cite{FIM} where $G$ is discrete and $\Gamma=S^1
\times \RR $ acts smoothly.

\section{Local action functions}

%\begin{lemma} Suppose $O$ is a principal $T$ orbit. Then there is
%a unique $T$-equivariant principal $G$ bundle on $O$ up to
%isomorphism.
%\end{lemma}

%\begin{proof}
% Given any element for a fiber, the
%$T$-action produces a section and hence an equivariant
%trivialization, i.e an isomorphism to the trivial bundle $O \times
%G$ with trivial $T$ action on the fiber $G$.
%\end{proof}

Let $X= X_{\sigma}$ where $\sigma \in \Xi$. For any sub-cone
$\delta \le \sigma$ we denote the corresponding $T$-orbit by
$O_{\delta}$. As $T$  is Abelian, the stabilizer of $x$ in $T$ is
the same  for all $x \in O_{\delta}$. Denote this stabilizer
subgroup by $T_{\delta}$.

%corresponding to a cone $\sigma$ by $X_{\sigma}$ and $O_{\sigma}$
%respectively. When $\sigma$ is a cone of top dimension,
%$X_{\sigma}$ is equivariantly bi-holomorphic to a linear
%representation space of $X$ of $T$.

 Suppose $\mathcal{E}$ is a $T$-equivariant principal $G$-bundle
over $X$. Assume that $\mathcal{E}$ is trivial (we will show in
Lemma \ref{lem:cone1} that this holds for suitable $G$).
 Let $s: X
\to \mathcal{E}$ be any holomorphic section. We encode the
$T$-action on $\mathcal{E}$ as follows.

\begin{defn} For any $x\in X$ and $t\in T$, define $\rho_s(x,t) \in G$ as
follows,
\begin{equation}\label{eq:rho1}
t s(x) = s(tx)\cdot \rho_s(x,t)
\end{equation}
We say that $\rho_s: X \times T \to G$ is a local action function.
\end{defn}

 Since the action of
$G$ on each fiber of $\mathcal{E}$ is free, it follows that
$\rho_s(x,t)$ is well-defined and holomorphic in $x$ and $t$.

 It is easy to check that if $s'(x)= s(x)\cdot \gamma(x)$ is another section, then
\begin{equation}\label{eq:rho2}
 \rho_{s'}(x,t) = \gamma(tx)^{-1} \rho_s (x,t) \gamma(x)
\end{equation}

\begin{lemma}\label{lem:rho3} For any $t_1$, $t_2$ in $T$,
$\rho_s(x, t_1t_2)= \rho_s(t_2x,t_1) \rho_s(x,t_2)$.
\end{lemma}

\begin{proof} Obviously, $t_1t_2s(x) = s(t_1t_2x) \cdot \rho_s(x,t_1t_2)$. On the
other hand, $$t_1t_2s(x) = t_1( s(t_2x) \cdot \rho_s(x, t_2 )) =
(t_1 s(t_2x)) \cdot \rho_s(x, t_2 )  = s(t_1t_2x) \cdot
\rho_s(t_2x, t_1) \rho_s(x,t_2).$$
\end{proof}

It follows that if $x \in O_{\delta}$ then the restriction
\begin{equation}\label{eq:rho4}
\rho_s(x, \cdot): T_{\delta} \to G \end{equation} is a group
homomorphism.

Lemma \ref{lem:rho3} implies that the value of $\rho_s$ on any
$T$-orbit $O_{\delta}$ may be determined from the value of
$\rho_s$ at a point $x \in O_\delta$. However a stronger statement
holds.

\begin{lemma}\label{lem:rho4} The map $\rho_s$ is completely determined by its
restriction $\rho_s(x_0,\cdot): T \to G$ at any point $x_0$ in the
principal $T$-orbit $O$.
\end{lemma}

\begin{proof} Let $\delta \neq \{0\}$ be a sub-cone of $\sigma$.
  Let $x_\delta$ be any point in $O_\delta$. Then
there exist a point $x_1 \in O$ and a one parameter subgroup $\lambda^r(z)$ of $T$ such
that $\lim_{ z\to 0} \lambda^r(z)x_1 = x_\delta $ (see
\cite{Ful}, section 2.3). Then using Lemma \ref{lem:rho3}, we get $ \rho_s(x_1,
t \lambda^r(z))= \rho_s(\lambda^r(z) x_1,t) \rho_s(x_1,
\lambda^r(z) )$. Hence $\rho_s(\lambda^r(z) x_1,t)=  \rho_s(x_1, t
\lambda^r(z))\rho_s(x_1, \lambda^r(z) )^{-1}$. Taking limit as $z$
approaches $0$, we get
\begin{equation}\label{eq:rho5}
\rho_s(x_\delta, t) = \lim_{z\to 0} \rho_s(x_1, t
\lambda^r(z))\rho_s(x_1, \lambda^r(z) )^{-1}. \end{equation}
Since $\rho_{s}(x_1, \cdot)$ is determined by $\rho_{s}(x_0, \cdot)$, the lemma follows.
\end{proof}

\begin{lemma}\label{lem:equiv} Let $X_1$ and $X_2$ be affine toric varieties.
 Let $\alpha: X_1 \to X_2$ be an isomorphism of $T$-spaces up to an automorphism $a:T\to T$, i.e.
$\alpha \circ t = a(t) \circ \alpha$. Suppose $\pi_i: \mathcal{E}_i \to X_i$ is a $T$-equivariant trivial principal
$G$-bundle for $i=1,2$.
 Let $\phi: \mathcal{E}_1 \to \mathcal{E}_2$ be an
isomorphism of $T$-equivariant principal $G$-bundles over $X$ compatible with $\alpha$ and $a$:
$$ \pi_2 \circ \phi = \alpha \circ \pi_1  \quad {\rm and} \quad  \phi \circ t = a(t) \circ \phi. $$
 Let $s_1$ be any section of
$\mathcal{E}_1$ and let $s_2$ be the section of $\mathcal{E}_2$
defined by $s_2(\alpha(x))= \phi(s_1(x))$ for $x \in X_1$.
 Then $\rho_{s_1}(x, t ) = \rho_{s_2}(\alpha(x),a(t))$ for every $x\in X_1$, $t \in
 T$. In particular, if $\alpha$ and $a$ are both identity then $\rho_{s_1} =
 \rho_{s_2}$.
%= \rho_{s_2} $ where $s_2= \phi(s_1)$.
\end{lemma}

\begin{proof} The lemma follows from the following calculation.

 $$ \begin{array}{l} s_2(a(t) \alpha(x))\cdot \rho_{s_2}(\alpha(x),a(t)) \\
 = a(t) s_2( \alpha(x)) =  a(t) \phi( s_1 (x) ) \\
 = \phi (t s_1(x) ) = \phi(s_1(tx)\cdot \rho_{s_1}(x,t)) \\
= \phi(s_1(tx)) \cdot \rho_{s_1}(x,t)  = s_2(\alpha(tx))\cdot \rho_{s_1}(x,t) \\
 = s_2(a(t) \alpha(x) ) \cdot \rho_{s_1}(x,t).  \end{array}$$
\end{proof}
% We can define an action of $T$ on the space of sections of $\mathcal{E}$
%by \begin{equation}(t\cdot s )(x) = t s(t^{-1}x). \end{equation}
%Note that
%\begin{equation}\label{eq:rho6}
%(t\cdot s )(x) = s(x) \cdot \rho_s(x, t^{-1})^{-1}.
%\end{equation}

\begin{lemma}\label{lem:rho5} If $\rho_s(x,\cdot)$ is independent of $x$, then
it defines a group homomorphism $\rho_s: T \to G$. Conversely if
$\rho_s(x_0,\cdot)$ is a group homomorphism for some $x_0 \in O$,
then $\rho_s(x,t)$ is independent of $x$.
\end{lemma}

\begin{proof} The first claim follows immediately from Lemma \ref{lem:rho3}.
On the other hand if $\rho_s(x_0,\cdot)$ is a group homomorphism,
then we have
$$ \rho_s(x_0, t) \rho_s(x_0, u) = \rho_s(x_0,tu) =  \rho_s(ux_0,t) \rho_s(x_0, u )   $$
for any $u,\,t \in T $. Therefore for any $u \in T$
$$ \rho_s(ux_0,\cdot) = \rho_s(x_0, \cdot).$$
Now the result follows either by using holomorphicity of
$\rho_s(x,t)$ in $x$, or equation \eqref{eq:rho5}.
\end{proof}

\begin{lemma}\label{lem:discrete} Suppose $G$ is a discrete group and
$X \cong \CC^d$. Let $\mathcal{E}$ be a $T$-equivariant holomorphic principal $G$-bundle
over $X$.  Then $\mathcal{E}$ is trivial and for any section $s$ of $\mathcal{E}$,
$\rho_s(x,\cdot):T \to G$ is independent of $x$ and is the trivial
homomorphism.
\end{lemma}

\begin{proof} Note that as $X$ is contractible, the bundle $\mathcal{E}$ is
topologically trivial. Therefore it admits a continuous section $s$. But as $X$ is
connected and $G$ is discrete, $s$ is constant and hence holomorphic.

 Fix $t \in T$. Since $X$ is connected and $\rho_s(x,t)$
is continuous is $x$, it is a constant map due to the discreteness
of $G$. This implies that $\rho_s(x,\cdot)$ is independent of $x$ and
therefore is a group homomorphism $\rho_s(\cdot): T \to G$ by
Lemma \ref{lem:rho5}. Continuity in $t$ and connectedness of $T$
then imply that $\rho_s(t)$ is a constant. This completes the
proof.
\end{proof}

%This shows that any $T$-action on a principal $G$-bundle over
%$X\cong \CC^n$ must be trivial if $G$ is a discrete group.

\section{Local action homomorphisms}

Suppose $\mathcal{E}$ is a holomorphic principal $G$-bundle over
$X \cong \CC^n$ where $G$ is a subgroup of $GL_k(\CC)$. Then by
Oka-Grauert theory \cite{Gra} $\mathcal{E}$ is trivial and admits
a holomorphic section $s$. We will show that if $G$ is Abelian and
closed in $GL_k(\CC)$, then $s$ can be chosen so that the local
action function $\rho_s$ is a homomorphism.

\begin{prop}\label{prop:lcC} Let $G$ be a subgroup of $GL_k(\CC)$
that is  closed in $GL_{k}(\CC)$. Suppose $f: \CC^* \to G$ is a
holomorphic map such that $\lim_{z \to 0} f(zt)f(z)^{-1}= I$  for
every $t \in \CC^*$. Then $f$ admits a holomorphic extension $f:
\CC \to G$.
\end{prop}

\begin{proof}
Let \begin{equation}\label{laurentf} f(z) =
\sum_{n=-\infty}^{\infty} C_n z^n
\end{equation}
be the Laurent series representation of $f$ for $z \in
\CC^{\ast}$, where each $C_n \in M_k(\CC)$. Note that $f(\cdot)^{-1}:
\CC^* \to G$ is also holomorphic. Let its Laurent series be
\begin{equation}\label{laurentfinv} f(z)^{-1} =
\sum_{n=-\infty}^{\infty} A_n z^n
\end{equation}
where each $A_n \in M_k(\CC)$. Then we have
 \begin{equation} f(zt) f(z)^{-1} = \sum_{n=-\infty}^{\infty} z^n t^n ( \sum_{m=-\infty}^{\infty}
 C_{n-m} A_{m} t^{-m} ).
\end{equation}
As the limit $\lim_{z \to 0} f(zt)f(z)^{-1}$ exists by assumption,
 %a simple application of Morera's
 the Riemann extension theorem shows that $ f(zt) f(z)^{-1} $ is
holomorphic at $z= 0$ for every $t \in \CC^*$.
% This follows from Morera�s theorem using arbitrarily small circle about zero
% and using boundedness  of f near zero.
Therefore,
\begin{equation}\label{eq:cn1}
\sum_{m=-\infty}^{\infty} C_{n-m} A_{m} t^{-m} = 0
\end{equation}
for every $n<0$ and every $t \in \CC^*$. This implies that
% by using formulas of Laurent coefficients for instance
\begin{equation}\label{eq:cn2}
 C_{n-m} A_{m} = 0
\end{equation}
for every $n<0$ and every $m$. Setting $j=n-m$ and varying $n$, we
obtain \begin{equation}\label{eq:cn3}
 C_{j} A_{m} = 0
\end{equation}
for every $j<- m$ for every $m$.

 Let $S$ be the set of all column
vectors of the matrices $A_m$, $m \in \ZZ$. Since every subspace
of $\CC^k$ is closed, the span of $S$ is closed in $\CC^k$.
Therefore the columns of $f(z)^{-1}$ belong to the span of $S$.
 Since $f(z)^{-1} \in GL_k(\CC)$, the span of $S$
  must equal $\CC^k$. Therefore there exists a finite set of
 values of $m$, say $m_1,\ldots, m_p$, such that the column spaces of the corresponding
 $A_{m}$'s span $\CC^k$. Let $m_0= {\rm min} \{-m_1, \ldots,-m_p \}
 $.  Consider any $j<  m_0 $. Then $j<- m_i$ for each $1\le i\le p$.
 Hence by equation \eqref{eq:cn3}
 $C_j A_{m_i} = 0$ for  each $1\le i\le p$. Hence $C_j = 0$.

 Assume that $n_0$ is the minimum value of $n$ such that $C_{n_0}$
 is nonzero. Then $f(z)= z^{n_0} \phi(z)$ where $\phi: \CC \to GL_k(\CC)$
 is holomorphic near zero. Then
 \begin{equation}\label{eq:asm} I = \lim_{z\to 0} f(zt)f(z)^{-1} = \lim_{z\to 0} \phi(zt)\phi(z)^{-1} t^{n_0}.
 \end{equation}
Since $\ds{\lim_{z\to 0} \phi(z) = C_{n_0}} $ and $\ds{\lim_{z\to
0} \phi(tz)= C_{n_0} }$, multiplying both sides of \eqref{eq:asm} by $\lim_{z \to 0} \phi(z) $ yields
%% we are using that radius of convergence of \phi is positive
 \begin{equation}  C_{n_0} =  \lim_{z\to 0} \phi(z) = \lim_{z\to 0} \phi(zt) t^{n_0}  = C_{n_0} t^{n_0}
 \end{equation}
 for every $t \in \CC^{*}$.
  Therefore we must have $n_0 =
 0$. Hence $f$ is holomorphic at $0$ and $f(0) =C_0$.

 Let $g(z) = \det(f(z))$. It follows that $g:\CC \to \CC $ is
 holomorphic and $\ds{\lim_{z\to 0} g(zt) g(z)^{-1} = 1}$. If $g$ has a
 zero of order $n$ at zero, we get $\lim_{z\to 0} g(zt) g(z)^{-1}
 = t^n$. Therefore $n=0$ and $g(0) \neq 0$. Hence $f(0)$ is nonsingular.
 Thus $f$ defines a holomorphic map $f:\CC \to GL_k$.

Moreover if $G$ is closed in $GL_k(\CC)$, it is evident that $f(0)
\in G$ so that $f$ defines a holomorphic map $f: \CC \to G $.
\end{proof}

\begin{theorem}\label{thm:homo} Consider the standard action of $T$ on $ \CC^n$.
Suppose $\mathcal{E}$ is a holomorphic $T$-equivariant
 principal  $G$-bundle over
$ \CC^n$ where $G$ is a closed Abelian subgroup of $GL_k(\CC)$.
Then there exists a holomorphic section $s'$ of $\mathcal{E}$ such
that $\rho_{s'}(x,\cdot): T \to G$ equals $\rho_{s'}(0,\cdot)$  for any
$x$ in $\CC^n$.
\end{theorem}

\begin{proof} Let $s$ be a holomorphic section of $\mathcal{E}$ over $\CC^n$.
  Fix a point $x_0 $ in $ O = T$. Note that $\rho_s(0,\cdot): T \to G$ is a
  group homomorphism as the origin $0$ is fixed by $T$. Define a holomorphic function $F :T \to G$ by
  \begin{equation}\label{eq:F} F(t) = \rho_s(0, t)^{-1} \rho_s(x_0, t ).
  \end{equation}
  We may therefore  write
 \begin{equation} \label{eq:F2} \rho_s(x_0, t )  =  \rho_s(0, t) F(t).
  \end{equation}

For any $y, x \in \CC^n$, define $yx$ to be coordinate-wise
multiplication of $y$ and $x$. Let $z$ denote an element of $T$.
   For any $y \in \CC^n$,
  $$  \begin{array}{l} \rho_s(yx_0,t) = \lim_{z\to y} \rho_s(zx_0, t) =  \lim_{z\to y} \rho_s(x_0, tz) \rho_s(x_0,
  z)^{-1}\\
     =  \lim_{z\to y} \rho_s(0, tz) F(tz)F(z)^{-1} \rho_s(0,
  z)^{-1}.
 \end{array} $$
 Therefore,
 \begin{equation}\label{eq:lim}  \lim_{z\to y}  F(tz)F(z)^{-1} = \lim_{z\to y}  \rho_s (0, tz)^{-1}  \rho_s(yx_0,t) \rho_s (0, z)
 = \rho_s(yx_0,t) \rho_s (0, t)^{-1},
 \end{equation}
 where the last equality follows from the assumption that $G$ is
 Abelian, and the fact that $\rho_s(0,\cdot)$ is a homomorphism.

Define $P_k = \CC^n - \cup_{i=k}^n \{z_i =0 \}$ where $1 \le k \le
n$. Note that $P_1 = (\CC^*)^n$ and $P_n = \CC^n$. We will now show by induction
over $k$ that $F$ admits a holomorphic $G$-valued extension over $\CC^n$. Note
that  $F$ admits an extension over $P_1$ by definition. Now
assume that $F$ has a $G$-valued holomorphic extension over $P_k$. Take any point $p=(p_1, \ldots,
p_n)\in P_k$. Note that $tp \in P_k$ for any $t \in T$. Setting
$y= p $ in \eqref{eq:lim}, we get
\begin{equation}\label{eq:limp}    F(tp)F(p)^{-1}
 = \rho_s(px_0,t) \rho_s (0, t)^{-1}.
 \end{equation}
 Let  $\pi_k: \CC^n \to \{
z_k = 0\}$ be the standard projection. Taking limit of
\eqref{eq:limp} as $p_k \to 0$, we get
\begin{equation}\label{eq:limq}   \lim_{p_k \to 0} F(tp)F(p)^{-1}
 = \rho_s( \pi_k(p)  x_0,t) \rho_s (0, t)^{-1}.
 \end{equation}

Let $t_k \in \CC^*$. Write $\iota(t_k)$ for $(1, \ldots,1,
t_k, 1, \ldots, 1)$ where $t_k $ occupies the $k$-th position.

In the case $p_i \neq 0$ for each $i \neq k$,
$t(p)=(p_1, \ldots, p_{k-1}, 1, p_{k+1}, \ldots, p_n)$ defines
an element of $T$.
 Then by Lemma \ref{lem:rho3}, and
using $G$ is Abelian, we have
\begin{equation} \begin{array}{l}
\rho_s(\pi_k(p) x_0, \iota(t_k))\\
= \rho_s(t(p)\pi_k(x_0),\iota(t_k) )\\
= \rho_s( \pi_k(x_0), \iota(t_k) t(p))\, \rho_s(\pi_k(x_0),t(p))^{-1} \\
= \rho_s(\pi_k(x_0), t(p) \iota(t_k) ) \, \rho_s(\pi_k(x_0),t(p))^{-1} \\
= \rho_s(\iota(t_k) \pi_k(x_0) , t(p))\, \rho_s(\pi_k(x_0), \iota(t_k)) \, \rho_s(\pi_k(x_0), t(p))^{-1} \\
= \rho_s(\pi_k(x_0), t(p)) \, \rho_s(\pi_k(x_0), \iota(t_k))\, \rho_s(\pi_k(x_0), t(p))^{-1} \\
%= \rho_s(e, t_k) \rho_s( e, t(p)) \rho_s(e, t(p))^{-1} \\
= \rho_s(\pi_k(x_0), \iota(t_k)).
\end{array}
\end{equation}
The set $\{\pi_k(p) x_0 |   p_i \neq 0 \, \forall i\neq k \}$
is dense in the hyperplane $z_k = 0$. Hence
 by holomorphicity of $\rho_s(w, \iota(t_k))$ in $w$, we obtain
$\rho_s(w, \iota(t_k))= \rho_s(\pi_k(x_0), \iota(t_k)) $ for every $w$ in
the hyperplane $z_k=0$. Therefore $\rho_s(w, \iota(t_k))$ is constant for
 $w \in \{z_k=0\}$ and equals $\rho_s(0,\iota(t_k)) $.
 In particular,
 \begin{equation}\label{eq:cons} \rho_s(\pi_k(p)x_0, \iota(t_k))=
\rho_s(0,\iota(t_k)) \end{equation}
for any $p$ in $P_k$, since $ \pi_k(p)x_0 \in \{z_k=0\}$.

Fix $p_i$ for each $i\neq k$, and define
\begin{equation}\label{eq:limq2} f(p_k)= F(p_1, \ldots, p_k,
\ldots, p_n).\end{equation}
 Then setting $t =\iota(t_k) $  in \eqref{eq:limq}, we get
\begin{equation}\label{eq:limf}  \lim_{p_k \to 0} f(t_k p_k) f(p_k)^{-1}=
   \rho_s( \pi_k(p) x_0, \iota(t_k)) \,
\rho_s(0,\iota(t_k))^{-1}. \end{equation}
Therefore by \eqref{eq:cons},
$$ \lim_{p_k \to 0} f(t_k p_k) f(p_k)^{-1}= I$$ for any $t_k \in \CC^*$.
 Therefore Proposition \ref{prop:lcC} applies to $f$. Hence
$$\lim_{p_{k} \to 0} F(p_1,\ldots, p_n) = \lim_{p_k \to 0} f(p_k)$$ exists and takes value in $G$.
 As this argument works for every $p \in P_k$, by the Riemann extension theorem,
$F$ has a unique $G$-valued extension over $P_{k+1}= \CC^n - \cup_{i=k+1}^n
\{z_i =0 \}$. Therefore, by induction, $F$ admits a unique $G$-valued holomorphic extension
over $\CC^n$.

Now define a new section $s'$ of $\mathcal{E} $ by
\begin{equation}
s'(zx_0) = s(zx_0)\cdot F(z)
\end{equation}
for every $z\in \CC^n$.
Note that $F(I) = I$ as $\rho_s(x_0,I) = I = \rho_s(0,I)$. Then
\begin{equation}
ts'(x_0) = t (s(x_0)\cdot F(I))  = t s(x_0)=  s(tx_0)\cdot
\rho_s(x_0,t).
\end{equation}
On the other hand,
\begin{equation}
ts'(x_0) = s'(tx_0) \cdot \rho_{s'}(x_0,t) = s(tx_0)\cdot F(t)
\rho_{s'}(x_0,t).
\end{equation}
Therefore, using \eqref{eq:F2}, we have
\begin{equation}
\rho_{s'}(x_0,t) = F(t)^{-1} \rho_s(x_0,t)  = F(t)^{-1}
\rho_s(0,t) F(t) = \rho_s(0,t).
\end{equation}
Therefore $\rho_{s'}(x_0,t)$ is a homomorphism and the theorem
follows from Lemma \ref{lem:rho5}.
\end{proof}

\begin{remark}\label{rem:indep} Note that the homomorphism $\rho_s(0,\cdot)$ is
independent of the choice of the section $s$ by \eqref{eq:rho2}
when $G$ is Abelian.
\end{remark}

\begin{lemma}\label{lem:cone1} Let $G$ be a closed Abelian subgroup
 of $GL_k(\CC)$. Suppose $\sigma$ is any cone of an $n$-dimensional
nonsingular fan $\Xi$ and $\mathcal{E}$ a $T$-equivariant
holomorphic principal $G$-bundle on the affine toric variety
$X_{\sigma}$. Then $\mathcal{E} $ is trivial and admits a section
$s^*$ for which the local action function $\rho_{s^*}$ is a
homomorphism. Moreover, there exists a canonical homomorphism
$\rho_{\sigma}: T_{\sigma} \to G$ and a choice of ${s^*}$ such
that $\rho_{s^*}(t) = \rho_{\sigma}(\pi_{\sigma}(t) )$, where
$\pi_{\sigma}: T \to T_{\sigma}$ is the projection associated to
the decomposition $T \cong T_{\sigma} \times O_{\sigma}$.
\end{lemma}

\begin{proof} Let $d$ denote the dimension of the cone $\sigma$.
Note that there exists an isomorphism $ \alpha : X_{\sigma} \to
\CC^{d} \times (\CC^*)^{n-d}$ where the latter space has the
standard $T$-action, such that $\alpha$ is equivariant up to an
automorphism $a_{\sigma}$ of $T$:
$$ \alpha(tx)= a_{\sigma}(t) \alpha(x).$$
Define $H= a_{\sigma}(T_{\sigma})$ and $K = a_{\sigma}(O_{\sigma}) $.
Note $H \cong (\CC^*)^d$ and $K \cong (\CC^*)^{n-d}$.

%consider $a_{\sigma}$ as a map $a_{\sigma}: T \to H \times K$.

Let $\mathcal{F}$ be the pull-back of $\mathcal{E}$   with respect
to $\alpha^{-1}$. Let $\phi: \mathcal{E} \to \mathcal{F}$ be the
natural isomorphism. Note that $\mathcal{F} $ inherits natural
actions of $T$ and $G$ that satisfy
$$ \phi(t e) = a_{\sigma}(t) \phi(e), \quad \phi(e\cdot g) = \phi(e) \cdot g$$
for every $e \in \mathcal{E}$. This makes $\mathcal{F} $ a
$T$-equivariant principal $G$-bundle over $\CC^{d} \times
(\CC^*)^{n-d} $.

 Fix a point $y_0$ in $(\CC^{*})^{n-d}$. Then by
Oka-Grauert theory $\mathcal{F}$ is trivial on $\CC^{d} \times
\{y_0\} $. Let $s$ be a section of this restricted bundle. We
extend $s$ to a section of $\mathcal{F}$ over $ \CC^{d} \times
(\CC^*)^{n-d}$ by defining
\begin{equation}\label{eq:extn}
 s(x,k y_0)= k s(x,y_0)
\end{equation}
for every $x\in \CC^d$ and $k \in K$. This shows that
$\mathcal{F}$, and consequently $\mathcal{E}$, is trivial.

By Theorem \ref{thm:homo} and Remark \ref{rem:indep} we may assume
that the local action function  of the section $s$ over $\CC^{d}
\times \{y_0\}$ satisfies \begin{equation} \rho_s((x,y_0),h) =
\rho_s((0,y_0),h)
\end{equation} for all $h \in H$, and defines a
homomorphism $ H \to G$ that is independent of $s$.

Since
\begin{equation} \begin{array}{l}
h s(x,k y_0) = hk s(x,y_0) = kh s(x, y_0) \\
 = k s(hx,y_0) \cdot
\rho_s((x,y_0),h) = s(hx,ky_0) \cdot \rho_s((x,y_0),h),
\end{array} \end{equation}
we deduce that
\begin{equation}
\rho_{s}((x,ky_0),h) = \rho_s((x,y_0),h)
\end{equation}
for every $k\in K$. This shows that the homomorphism $\rho_s : H
\to G$ is  independent of the choice of $y_0$ as well. Recall that
$O_{\sigma} \cong (\CC^*)^{n-d}$ and $a_{\sigma}(T_{\sigma}) = H$.
We define $\rho_{\sigma}: T_{\sigma} \to G$ by
\begin{equation}\label{eq:rhosigma} \rho_{\sigma} (t) =
\rho_s((0,y_0),a_{\sigma}(t)).
\end{equation}

It follows easily from \eqref{eq:extn} that $ks(x,y) = s(x,ky)$
for any $y \in (\CC^*)^{n-d}$. Then,
\begin{equation} \begin{array}{l}
hk s(x,y) = h s(x,ky) = s(hx, ky) \cdot \rho_s((x,ky),h) \\
= s(hx, ky) \cdot \rho_s((x,y_0),h) = s(hx, ky) \cdot
\rho_s((0,y_0),h) \end{array}
\end{equation}
for every $(h,k) \in H \times K$ and $(x,y) \in \CC^d \times
(\CC^*)^{n-d}$. Therefore,
\begin{equation}\label{eq:abo}
\rho_s((x,y), hk ) = \rho_s((0,y_0), h).
\end{equation}
Then  $\rho_s((x,y), \cdot )$ is a homomorphism as
$\rho_s((0,y_0), \cdot) $ is so.

 We set $s^{*}= \phi^{-1}(s)$. Then by Lemma \ref{lem:equiv} and \eqref{eq:abo},
\begin{equation}\label{eq:abo2}
\rho_{s^*}(x, t ) = \rho_s((\alpha(x), a_{\sigma}(t)) =
\rho_s((0,y_0), pr(a_{\sigma}(t))),
\end{equation}
 where $pr: T \to H$ is the
projection corresponding to the decomposition $T = H \times K $.
Then by definition of $H$ and $K$ we have $ pr \circ a_{\sigma} =
a_{\sigma}\circ \pi_{\sigma}$. Therefore from \eqref{eq:abo2} and
\eqref{eq:rhosigma} we have
$$ \rho_{s^*}(x, t ) =  \rho_s((0,y_0), a_{\sigma}(\pi_{\sigma}(t)))= \rho_{\sigma}(\pi_{\sigma}(t) ).$$
\end{proof}

%\begin{defn} Denote the homomorphism $\rho_{\sigma} \circ \pi_{\sigma} : T \to G$
%by $\widetilde{\rho}_{\sigma}$.
%\end{defn}

\section{Gluing condition}
The main idea of this section is borrowed from \cite{Kly}.

Let $X$ be a nonsingular toric variety of dimension $n$
corresponding to a fan $\Xi$.
 Suppose $\mathcal{E}$ is a
$T$-equivariant principal $G$-bundle over $X$ where $G$ is a
closed Abelian subgroup of $GL_k(\CC)$.

Let $\sigma $ be any maximal cone in $\Xi$.
 Let $\widetilde{\rho}_{\sigma} = \rho_{\sigma} \circ \pi_{\sigma} :T \to G$
  where $\rho_{\sigma}: T_{\sigma} \to G$ is the homomorphism obtained by
   applying Lemma \ref{lem:cone1} to the bundle
$\mathcal{E}_{\sigma} := \mathcal{E}|_{X_{\sigma}}$.  Let
$s_{\sigma}$ a section of $\mathcal{E}_{\sigma}$ whose local
action homomorphism is $\widetilde{\rho}_{\sigma}$.

 Let $\psi_{\sigma}: \mathcal{E}_{\sigma} \to
X_{\sigma} \times G $ be the trivialization induced by the section
$s_{\sigma}$,
$$  \psi_{\sigma} (s_{\sigma}(x) \cdot h ) = (x, h). $$
Note that
$$ \psi_{\sigma} (t s_{\sigma}(x) \cdot h ) = (tx, \widetilde{\rho}_{\sigma}(t) h).$$
So the $T$ action on the trivialization $X_{\sigma} \times G$ is
defined by $$t(x,h) = (tx, \widetilde{\rho}_{\sigma}(t)h ). $$

 Let $\sigma$, $\tau$ be any two maximal cones. Let
$\phi_{\tau \sigma}: X_{\sigma} \cap X_{\tau} \to G $ denote the
transition function defined as follows,  $$\psi_{\tau}
\psi_{\sigma}^{-1} (x,h) = (x, \phi_{\tau \sigma}(x) h). $$
 By equivariance, we  have
$$ t (\psi_{\tau} \psi_{\sigma}^{-1}(x, h)) = \psi_{\tau} \psi_{\sigma}^{-1} (t(x,h)).  $$
This implies that
$$ (tx, \widetilde{\rho}_{\tau}(t) \phi_{\tau \sigma}(x) h ) = (tx, \phi_{\tau \sigma}(tx)\widetilde{\rho}_{\sigma}(t)h ). $$
Therefore,
\begin{equation}
\phi_{\tau \sigma}(tx)= \widetilde{\rho}_{\tau}(t) \phi_{\tau
\sigma}(x) \widetilde{\rho}_{\sigma}(t)^{-1} = \phi_{\tau
\sigma}(x)
\widetilde{\rho}_{\tau}(t)\widetilde{\rho}_{\sigma}(t)^{-1}.
\end{equation}

Consider the point $x_0=(1, \ldots, 1)$ in the principal $T$-orbit
 $O$ of $X$. For any maximal cones $\tau$ and $\sigma$, we have
$$ \widetilde{\rho}_{\tau}(t) \widetilde{\rho}_{\sigma}(t)^{-1} =  \phi_{\tau \sigma}(x_0)^{-1} \phi_{\tau \sigma}(tx_0).$$
Therefore the function $\widetilde{\rho}_{\tau}
\widetilde{\rho}_{\sigma}^{-1} :T \to G$ admits a holomorphic
extension to a function from $X_{\tau} \cap X_{\sigma}$ to $G$,
namely $ \phi_{\tau \sigma}(x_0)^{-1} \phi_{\tau \sigma}(\cdot) $.

\begin{theorem}\label{thm:clas} Let $X$ be an $n$-dimensional nonsingular toric
variety with fan $\Xi$.  Let $G$ be a closed Abelian subgroup of
$GL_k(\CC)$. Then the isomorphism classes of $T$-equivariant
holomorphic principal $G$-bundles on $X$ are in one-to-one
correspondence with collections of holomorphic group homomorphisms
$\{\rho_{\sigma}: T_{\sigma} \to G\, |\, \sigma \,{\rm is \; a\;
maximal \; cone\; of \;} \Xi \}$ which satisfy the extension
condition: Each $ (\rho_{\tau} \circ \pi_{\tau}) (\rho_{\sigma}
\circ \pi_{\sigma})^{-1}$ extends to a $G$-valued holomorphic
function over $X_{\sigma} \cap X_{\tau}$.
\end{theorem}

\begin{proof} Given a $T$-equivariant principal $G$-bundle $\mathcal{E}$ on $X$,
we have a canonical collection of homomorphisms $\{ \rho_{\sigma}
: T_{\sigma} \to G \}$ by Lemma \ref{lem:cone1}. We have shown
above that this collection satisfies the extension condition.
Moreover, the collection of homomorphisms is invariant under an
isomorphism of the bundle by Lemma \ref{lem:equiv}.

Conversely given a collection of homomorphisms $\{\rho_{\sigma}
\}$ satisfying the extension condition, define
$\widetilde{\rho}_{\sigma} = \rho_{\sigma} \circ \pi_{\sigma}$.
Let $\phi_{\tau \sigma}: X_{\sigma}\cap X_{\tau} \to G $ denote
the extension of
 $\widetilde{\rho}_{\tau} \widetilde{\rho}_{\sigma}^{-1}$. Note that $\{\phi_{\tau \sigma} \}$
 satisfies the cocycle condition. Therefore we may  construct
 a principal $G$-bundle $\mathcal{E}$ over $X$ with  $\{\phi_{\tau \sigma} \}$
 as transition functions,
 $$ \mathcal{E} = (\bigsqcup_{\sigma} X_{\sigma} \times G )/ \sim$$
where $(x,g) \sim (y, h)$ for $(x,g) \in X_{\sigma} \times G $ and
$(y,h) \in X_{\tau} \times G $ if and only if
\begin{equation}\label{eq:equiv} x=
y, \; x \in X_{\sigma}\cap X_{\tau} \; {\rm and}\;  h = \phi_{\tau
\sigma}(x) g.\end{equation}

 Define $T$ action on each $X_{\sigma} \times G$ by
 $t(x,g)= (tx, \widetilde{\rho}_{\sigma}(t) g)$. Then note that if $(y,h) \in X_{\tau} \times
 G$ is equivalent to $(x,g) \in X_{\sigma} \times G $, then
 \begin{equation}\label{eq:equiv2}
t(y,h) = t(x, \phi_{\tau \sigma }(x) g) =
 (tx, \widetilde{\rho}_{\tau}(t) \phi_{\tau \sigma }(x) g ).\end{equation}
 Now if $x $ belongs to the open orbit $O =T \subset X_{\sigma} \cap
 X_{\tau}$, then
\begin{equation}\label{eq:equiv3}
 \phi_{\tau \sigma}(tx) \widetilde{\rho}_{\sigma}(t)  = \widetilde{\rho}_{\tau}(tx) \widetilde{\rho}_{\sigma}(tx)^{-1}
 \widetilde{\rho}_{\sigma}(t)
 = \widetilde{\rho}_{\tau}(t) \widetilde{\rho}_{\tau} (x) \widetilde{\rho}_{\sigma} (x)^{-1}  =
 \widetilde{\rho}_{\tau}(t) \phi_{\tau \sigma}(x).\end{equation}

 Since both $\phi_{\tau \sigma}(tx) \widetilde{\rho}_{\sigma}(t) $ and $ \widetilde{\rho}_{\tau}(t) \phi_{\tau
 \sigma}(x)$ are continuous in $x$ on $X_{\sigma} \cap X_{\tau}$
 and $O$ is dense in $X_{\sigma} \cap X_{\tau} $,
\begin{equation}\label{eq:equiv4}
\phi_{\tau \sigma}(tx) \widetilde{\rho}_{\sigma}(t) =
\widetilde{\rho}_{\tau}(t) \phi_{\tau \sigma}(x)\; {\rm for\;
all}\; x \in X_{\sigma} \cap X_{\tau}.
\end{equation}
From \eqref{eq:equiv2} and \eqref{eq:equiv4}, we have
$$t(y,h) = (tx, \phi_{\tau \sigma}(tx) \widetilde{\rho}_{\sigma}(t) g   )$$
 whenever $(x,g) \sim (y,h)
 $. Since $t(x,g)= (tx, \widetilde{\rho}_{\sigma}(t) g)$, by \eqref{eq:equiv} this implies that $t(y,h)\sim t(x,g) $
whenever $(x,g) \sim (y,h)
 $.  In other words, the $T$-actions on the $X_{\sigma} \times G$
 are compatible and define an action on $\mathcal{E}$. It is
 obvious that $\{\rho_{\sigma}\}$ are the local homomorphisms
 associated to $\mathcal{E}$.
\end{proof}

\begin{theorem}\label{thm:discrete} If $G$ is a discrete group, then
any holomorphic $T$-equivariant principal $G$-bundle $\mathcal{E}$ over a
nonsingular toric variety is trivial with trivial $T$-action.
\end{theorem}

\begin{proof} By using Lemma \ref{lem:discrete} and mimicking the proof of
Lemma \ref{lem:cone1}, we obtain that $\mathcal{E}$ is trivial over any
nonsingular affine toric variety $X_{\sigma}$. We also obtain canonical
homomorphisms $\rho_{\sigma}:T_{\sigma} \to G$. Then, we mimic the proof
of Theorem \ref{thm:clas} and obtain an analogous result.
 Note that a holomorphic homomorphism  from any $T_{\sigma}$ to the discrete group $G$ must
be trivial. The result follows.
\end{proof}

\subsection{Algebraic case}
%\begin{prop}\label{prop:lcCA} Let $G$ be a subgroup of $GL_k(\CC)$
%that is  closed in $GL_{k}(\CC)$. Suppose $f: \CC^* \to G$ is a
%regular map such that $\lim_{z \to 0} f(zt)f(z)^{-1}= I$  for
%every $t \in \CC^*$. Then $f$ admits a regular extension $f: \CC
%\to G$.
%\end{prop}

%\begin{proof}  A priori $f$ is represented by
%$k \times k$ matrix $A$ with entries in the ring $\CC[z, z^{-1}]$.
%As $G$ is closed in the Zariski topology, it is closed in the
%analytic topology. Therefore, by Proposition \ref{prop:lcC}, $f$
%admits a holomorphic extension over $\CC$. Hence the entries of
%$A$ are convergent power series in $z$. Therefore, they must be
%polynomials in $z$.
%\end{proof}

Suppose that $\mathcal{E}$ is a $T$-equivariant algebraic
principal $G$-bundle over $\CC^n$ where $G$ is a closed  subgroup
of $GL_k(\CC)$.
 It follows from \cite{Rag} that
 $\mathcal{E}$ is trivial. Then the local action
function corresponding to a section $s$ of $\mathcal{E}$ is a
regular map $\rho_s: \CC^n \times T \to G$. Therefore, the
function $F: T
 \to G$ of \eqref{eq:F} is regular.

 As $G$ is closed in the Zariski topology, it is closed in the
analytic topology. So, the proof of Theorem
 \ref{thm:homo} shows that $F$ admits a holomorphic extension
 $F: \CC^n \to G$. A priori $F$ is represented by
$k \times k$ matrix $A$ with entries in the ring $\CC[z_i,
z_i^{-1}, 1\le i \le n ]$. But as $F$ is holomorphic the entries
of $A$ are also convergent power series in $z_1, \ldots, z_n$.
Therefore the entries must be polynomials. Hence the extension $F:
\CC^n \to G$ is regular.
 This yields an algebraic analogue of Theorem
 \ref{thm:homo} which in turn leads to an algebraic analogue of
 Lemma \ref{lem:cone1}. Then the same proof as in the holomorphic
 case gives the following result.

\begin{theorem}\label{thm:clasA} Let $X$ be an $n$-dimensional nonsingular toric
variety with fan $\Xi$.  Let $G$ be a closed Abelian subgroup of
$GL_k(\CC)$. Then the isomorphism classes of $T$-equivariant
algebraic principal $G$-bundles on $X$ are in one-to-one
correspondence with collections of algebraic group homomorphisms
$\{\rho_{\sigma}: T_{\sigma} \to G\, |\, \sigma \,{\rm is \; a\;
maximal \; cone\; of \;} \Xi \}$ which satisfy the extension
condition: Each $ (\rho_{\tau} \circ \pi_{\tau}) (\rho_{\sigma}
\circ \pi_{\sigma})^{-1}$ extends to a $G$-valued regular function
over $X_{\sigma} \cap X_{\tau}$.
\end{theorem}

\section{Applications} Every complete nonsingular toric variety  $X$
admits an equivariant principal $G$-bundle: Let $\rho: T \to G$ be
a homomorphism. Set $\rho_{\sigma} = \rho$ for every $\sigma \in
\Xi(n)$. Then the extension condition is satisfied as $\rho_{\tau}
\rho_{\sigma}^{-1}$ is the identity map.
However, $X$ may admit
more equivariant principal $G$-bundles.

\begin{theorem}\label{thm:last1} Suppose $G$ is a closed Abelian subgroup of
$GL_k(\CC)$. Then
 an equivariant  principal $G$-bundle over a nonsingular
toric variety $X$ admits a reduction of structure group to the
intersection of $G$ with a maximal torus of $GL_k(\CC)$.
\end{theorem}

\begin{proof} Let $\sigma_1, \ldots, \sigma_m$ be the maximal
cones of the fan of $X$. A homomorphism $\rho_{\sigma_i}: T \to G$
may be regarded as a representation of $T$ on $V:=\CC^k$. We
denote $\rho_{\sigma_i}$ by $\rho_i$.

The representation $\rho_i$  decomposes into one dimensional
representations with not necessarily distinct characters. However
we may collate all one dimensional representations having the same
character $\lambda_i$ into a subspace $V(\lambda_i)$, and write $V
= \oplus_{\lambda_i} V(\lambda_i) $ where $\lambda_i$ varies over
characters of $\rho_i$.

 As $\rho_{1}$ and $\rho_{2}$ commute, $V(\lambda_1)$ is invariant under
 $\rho_2$. Let $V(\lambda_1, \lambda_2)$ denote the direct sum of
 the one dimensional subspaces of $V(\lambda_1)$ which are irreducible
 components of $\rho_2\mid_{V(\lambda_1)}$ with character $\lambda_2$.
 Therefore $$V= \bigoplus_{\lambda_1,\lambda_2} V(\lambda_1, \lambda_2).$$

 It is easy to observe that $V(\lambda_1, \lambda_2) = V(\lambda_1) \cap
 V(\lambda_2)$. Therefore it is invariant under $\rho_3$.
Proceeding inductively, we have
 $$ V = \bigoplus_{\lambda_1, \ldots, \lambda_m} \, V(\lambda_1)\cap \ldots \cap V(\lambda_m).$$
Any one dimensional subspace of $V(\lambda_1)\cap \ldots \cap
V(\lambda_m)$ is an eigenspace of $\rho_i(t)$ for every $i$ and
$t$. Therefore all these operators are simultaneously
diagonalizable. So there exists a $g \in GL(V)$ such that $g^{-1}
\rho_i(t) g \in (\CC^{*})^k \subset GL_k(\CC)$. In other words,
each $\rho_i(t)$ belongs to the intersection of the torus
$g(\CC^*)^k g^{-1}$ with $G$. By closedness of $G$ and $(\CC^*)^k$
in $GL_k(\CC)$, the transition functions $\rho_i \rho_j^{-1}$ also
belong to this intersection.
\end{proof}

Note that the above theorem implies that any $T$-equivariant vector bundle
with an abelian structure group splits equivariantly into a sum of line bundles.

A more precise result is obtained if we use the classification of
Abelian subgroups of $GL_k(\CC)$ in \cite{AM}: Let $\eta =(k_1,
\ldots, k_r)$ be a partition of $k$ into positive integers. Let
$K_{\eta}^*$ be the subgroup of $GL_k(\CC)$ consisting of all
block diagonal matrices such that the $i$-th block has size $k_i
\times k_i$, and every block is lower triangular with identical
nonzero diagonal elements. The diagonal elements of different
blocks need not be the same. Then given an Abelian subgroup $G$ of
$GL_k(\CC)$, there exists an element $h \in GL_k(\CC)$ and a
partition $\eta$ of $k$ such that $hGh^{-1}$ is a subgroup of
$K_{\eta}^*$.

\begin{lemma} Let $L$ denote the subgroup of lower triangular matrices
in $GL_p(\CC)$ having equal diagonal elements. Given any
homomorphism $\rho: \CC^* \to L $, denote the $(i,j)$-th entry of
the matrix $\rho(t)$ by $\rho_{ij}(t)$. Write $\mu(t)$ for the
common value of $\rho_{ii}(t)$. Then $\mu: \CC^* \to \CC^*$ is a
homomorphism and $\rho_{ij}(t) = 0$ if $i\neq j$.
\end{lemma}

\begin{proof} Comparing the $(1,1)$ entries of $\rho(t_1 t_2)$
and $\rho(t_1) \rho(t_2)$, we have $\mu(t_1 t_2) = \mu(t_1)
\mu(t_2)$. Hence $\mu(t) = t^{a}$ for some $a\in \ZZ$.

Similarly, comparing $(1,2)$ entries, we  have
\begin{equation}\label{eq:homo1}
\rho_{12}(t_1 t_2) = \rho_{12}(t_1) \mu(t_2) + \mu(t_1)
\rho_{12}(t_2).
\end{equation}

For simplicity, write $g(t)$ for $\rho_{12}(t)$.
 Differentiating \eqref{eq:homo1} with respect to
$t_1$ holding $t_2$ constant, we get
\begin{equation}\label{eq:homo2}
g'(t_1 t_2) t_2 = g'(t_1) \mu(t_2) + \mu'(t_1) g(t_2).
\end{equation}
Setting $t_2=1$ in \eqref{eq:homo2} and simplifying, we have
 \begin{equation}\label{eq:homo3}
\mu'(t_1) g(1) = 0.
\end{equation}
If $\mu'(t_1) = 0$, then $\mu(t) =1$ for all $t$. Then from
\eqref{eq:homo1} we have $g(t_1t_2) = g(t_1) + g(t_2)$. This
implies that $g(t) = c \ln(t)$ and by holomorphicity of $g$ on
$\CC^*$, $c$ must be zero. Therefore in this case $g(t) = 0$.

If $\mu'(t_1) \neq 0$, then $g(1)= 0$. Setting $t_1 =1$, $t= t_2$
and using $\mu(t) = t^{a}$ in \eqref{eq:homo2}, we get
 \begin{equation}\label{eq:homo3}
g'( t) t = g'(1) t^{a} + a g(t).
\end{equation}
Solving this differential equation, and using $g(1) = 0$, we
obtain $g(t) = g'(1) t^{a} \ln(t) $. For $g$ to be holomorphic on
$\CC^*$, we must have $g'(1) = 0$ and $g(t) = 0$. Thus
$\rho_{12}(t) = g(t)= 0$.

Suppose $a > b$.
If $\rho_{ij}(t)= 0$ for all $(i,j) $ such that $i \neq j$, and either 
$j < b$ or $ j= b$ but $i< a$, then comparing
 $(a,b)$ terms of $\rho(t_1 t_2)$ and $\rho(t_1) \rho(t_2)$, we
have
\begin{equation}\label{eq:homo2}
\rho_{ab}(t_1 t_2) = \rho_{ab}(t_1) \mu(t_2) + \mu(t_1)
\rho_{ab}(t_2).
\end{equation}
Then by analogy with equation \eqref{eq:homo1} we get
$\rho_{ab}(t)= 0$. Therefore, the lemma follows by induction.
\end{proof}

\begin{corollary}\label{cor:diag}  If $\rho: \CC^* \to G \subset h^{-1} K_\eta^* h$  is a homomorphism, then the image of $\rho$ 
is contained in the intersection $ h^{-1} (\CC^*)^k h \cap G$.   
\end{corollary}

\begin{theorem}\label{thm:last2}  Suppose $G$ is a closed Abelian subgroup of
$GL_k(\CC)$ that is contained in $h^{-1}K_{\eta}^* h$. Then
 an equivariant  principal $G$-bundle over a nonsingular
toric variety $X$ admits a reduction of structure group to the
intersection of $G$ with the torus $  h^{-1} (\CC^*)^k h $.
\end{theorem}

\begin{proof} It follows easily from Corollary \ref{cor:diag} that  if $\rho: T \to G \subset h^{-1} K_\eta^* h$ 
 is a homomorphism, then the image of $\rho$ is contained in $ h^{-1} (\CC^*)^k h \cap G$. Apply this to the homomorphisms
$\rho_{\sigma}$ of Theorem \ref{thm:clas} (or Theorem \ref{thm:clasA}). By closedness of G and $  h^{-1} (\CC^*)^k h $, the 
transition maps $\rho_{\sigma} \rho_{\tau}^{-1}$ also take values in  $ h^{-1} (\CC^*)^k h \cap G$.
\end{proof}

\begin{corollary}\label{cor:last} Suppose $G $ is a conjugate of the subgroup $K_{\eta}^{*}$ in $GL_{k}(\CC)$. 
Let $X$ be a complete nonsingular toric variety with fan $\Xi$.
 Then the isomorphism classes of $T$-equivariant principal $G$-bundles over $X$ is parametrized by $\ZZ^{dr}$ where 
$d$ is the number of one dimensional cones of $\Xi$ and $r$ is the number of summands in the partition $\eta$.  
\end{corollary}

\begin{proof} We may assume without loss of generality that $G= K_{\eta}^*$. 
A homomorphism $\rho_{\sigma}:T \to G$ is determined by $r$ homomorphisms $\mu_{\sigma}^{i}: T \to \CC^*$ corresponding to the 
diagonal blocks. Moreover the extension condition on $\rho_{\sigma} \rho_{\tau}^{-1}$ in Theorem \ref{thm:clas} or Theorem \ref{thm:clasA}
splits into similar condition on $\mu^{i}_{\sigma} (\mu^{i}_{\tau})^{-1}$ for each block. Therefore we consider each block separately.

For a particular block we consider homomorphisms $\{\mu_{\sigma} : T \to \CC^* \mid \sigma \in \Xi(n) \}$ that satisfy the extension condition.
Such a collection corresponds to an isomorphism class of $T$-equivariant principal $\CC^*$-bundle on $X$.  We may identify $T$-equivariant 
principal $\CC^*$-bundles on $X$ with  $T$-equivariant line bundles on $X$. The latter have been studied by Oda \cite{Oda}, for instance.
In Proposition 2.1(i) of \cite{Oda}, Oda defines a homomorphism from the group of support functions to the group of isomorphism classes
of equivariant line bundles. The homomorphism is injective when $X$ is complete. Moreover, when $X$ is nonsingular, the group of support functions is isomorphic to $\ZZ^d$, and  Theorem \ref{thm:clas} (or Theorem \ref{thm:clasA}) shows that the above homomorphism is surjective. 
More precisely, we can set $\mu_{\sigma}$ equal to the character corresponding to $l_{\sigma}$
 in Oda's construction, when $\sigma$ has maximal dimension. For lower dimensional cones $\tau$, the $l_{\tau}$'s are obtained by restriction 
of $l_{\sigma}$ where $\tau \subset \sigma$.
 Completeness of the fan and the extension condition ensures that we obtain a support function by this procedure.    
\end{proof}

%A  homomorphism from $T= (\CC^*)^2$ to $G$ is given by,
%$$ \mu(t_1,t_2) =  t_1^{a} t_2^{b},\quad g(t_1,t_2)= 0 $$ where $a, b \in \ZZ$.

% Explanation: $g(t_1, t_2 ) = g(t_1) \mu(t_2) + \mu(t_1) g(t_2)$
% Differentiating, holding $t_2$ constant, $g^'(t_1, t_2) t_2 = g^'(t_1) \mu(t_2) + \mu^'(t_1) g(t_2)$
%Putting $t_2= 1$, we have $ \mu^'(t_1) g(1) = 0$. If $\mu^'=0$, then $\mu(t)=1$ and $g(ts) = g(t) +g(s)$
% so that $g(t) = 0$ as $\ln(t)$ is not holomorphic. If $g(1)= 0$,

% $g(xy) = g(x) + g(y) $. Treating x as a constant and differentiating we get
% $g'(xy)x = g'(y) $. Now set $y=1$ to get $g'(x) = A \frac{1}{x}$. Together
% with $g(1)= 0$ we get $ g(x) = A \ln x$.
%

\medskip

 {\bf Acknowledgement.} It is a pleasure to thank Indranil Biswas, Paul Bressler, Mikiya Masuda, Sean Paul, Florent
Schaffhauser, V. Uma, and Mauricio Velasco for discussions and comments. Both authors thank the
 Universidad de los Andes for financially supporting this project. The first author also
  thanks the Commission for Developing Countries (CDC)
  of the International Mathematical Union (IMU) for providing financial support for his visit to
  the Universidad de los Andes, Bogot\'a in 2012 which initiated the work.

\renewcommand{\refname}{References}

%\section{Abstract local action functions}

%\begin{defn} A map $\rho: X \times T \to G$ satisfying
%$$ \rho(x, t_1 t_2) = \rho(t_2x, t_1) \rho( x, t_2) $$
%will be called an abstract local action function (a.l.a.f.).
%\end{defn}

%From the proof of Lemma \ref{lem:rho4}, it is evident that an
%$a.l.a.f.$ is determined by its restriction to any point in the
%principal $T$-orbit.

%\begin{lemma}
%Any $a.l.a.f.$ $\rho$ on $X$ gives rise to a $T$-action on a
%principal $G$-bundle $\mathcal{E}$ over $X$. The action is unique
%up to equivariant isomorphism.
%\end{lemma}

%\begin{proof}
%Since $X$ is contractible, by Oka-Grauert theory $\mathcal{E}$ is
%trivial and admits a section, say $s$. Then define $T$ action by
%$$ t s(x) \cdot g = s(tx)\cdot \rho(x,t)  g. $$

%The choice of $s$ in defining the action does not affect the
%equivariant isomorphism class of $\mathcal{E}$: If $s_1$ and $s_2$
%are two sections, and $\mathcal{E}_1$ and $\mathcal{E}_2$ denote
%the $T$-equivariant $G$-bundle structures on $\mathcal{E}$ induced
%by them, then the map $\phi: \mathcal{E}_1 \to \mathcal{E}_2$
%defined by $\phi(s_1(x)\cdot g)= s_2(x) \cdot g $ is an
%isomorphism of $T$-equivariant $G$-bundles.
%\end{proof}

\vspace{1cm}

%\vfill

\end{document}